\documentclass[letterpaper, 10pt, conference]{ieeeconf}
\IEEEoverridecommandlockouts

\usepackage{graphicx}
\usepackage{epsfig}
\usepackage{times}
\usepackage{amsmath}
\usepackage{amssymb}
\usepackage{cleveref}
\usepackage{subcaption,graphicx}
\usepackage{url}
\usepackage{color}
\usepackage{mathrsfs}
\usepackage{multirow}
\usepackage{caption}
\usepackage{subcaption}
\usepackage{cite}
\usepackage{wrapfig}
\usepackage{float}
\usepackage{cite}

\newtheorem{theorem}{Theorem}

\newtheorem{assumption}[theorem]{Assumption}

\newtheorem{definition}[theorem]{Definition}

\newtheorem{remark}[theorem]{Remark}

\newtheorem{thm}{Theorem}

\DeclareMathOperator*{\infm}{inf}
\DeclareMathOperator*{\minimize}{minimize}

\DeclareMathOperator*{\subjt}{subject \ to}
\DeclareMathOperator*{\sgn}{sign}

\newcommand{\bi}{\begin{itemize}}
\newcommand{\ei}{\end{itemize}}
\newcommand{\bd}{\begin{displaymath}}
\newcommand{\ed}{\end{displaymath}}
\newcommand{\be}{\begin{eqnarray*}}
\newcommand{\ee}{\end{eqnarray*}}

\usepackage[nodisplayskipstretch]{setspace}
\title{Feedback Stabilization Using Koopman Operator}
\author{Bowen Huang, Xu Ma, and Umesh Vaidya \\
\thanks{Financial support from the National Science Foundation grants CNS-1329915 and ECCS-1150405 is gratefully acknowledged. The authors are with the Department of Electrical and Computer Engineering at Iowa State University. 
{\tt\small Emails:\{bowen, maxu, ugvaidya\}@iastate.edu}. }
}

\begin{document}
\maketitle

\begin{abstract}
In this paper, we provide a systematic approach for the design of stabilizing feedback controllers for nonlinear control systems using the Koopman operator framework. The Koopman operator approach provides a linear representation for a nonlinear dynamical system and a bilinear representation for a nonlinear control system. The problem of feedback stabilization of a nonlinear control system is then transformed to the stabilization of a bilinear control system. We propose a control Lyapunov function (CLF)-based approach for the design of stabilizing feedback controllers for the bilinear system. The search for finding a CLF for the bilinear control system is formulated as a convex optimization problem. This leads to a schematic procedure for designing CLF-based stabilizing feedback controllers for the bilinear system and hence the original nonlinear system. Another advantage of the proposed controller design approach outlined in this paper is that it does not require explicit knowledge of system dynamics. In particular, the bilinear representation of a nonlinear control system in the Koopman eigenfunction space can be obtained from time-series data. Simulation results are presented to verify the main results on the design of stabilizing feedback controllers and the data-driven aspect of the proposed approach.     
\end{abstract}

\section{Introduction}
Providing a systematic procedure for the design of stabilizing feedback control for a general nonlinear system will have a significant impact on a variety of application domains. The lack of proper structure for a general nonlinear system makes this design problem challenging. There have been several attempts to provide such a systematic approach, including convex optimization-based Sum of Square (SoS) programming \cite{SOS_book, Parrilothesis} and differential geometric-based feedback linearization control \cite{sastry2013nonlinear,astolfi2015feedback}. The introduction of operator theoretic methods from the ergodic theory of dynamical systems provides another opportunity for the development of systematic methods for the design of feedback controllers \cite{Lasota}. The operator theoretic methods provide a linear representation for a nonlinear dynamical system. This linear representation of the nonlinear system is made possible by shifting focus from state space to space of functions using two linear and dual operators, namely, the Perron-Frobenius (P-F) and Koopman operators. The work involving the third author \cite{VaidyaMehtaTAC, Vaidya_CLM,raghunathan2014optimal} was the first to propose a systematic linear programming-based approach involving transfer P-F operator for the optimal control of nonlinear systems. This contribution was made possible by exploiting not only the linearity but also the positivity and Markov properties of the P-F operator. 

More recently, there has been increased research activity on the use of Koopman operator for the analysis and control of nonlinear systems \cite{Meic_model_reduction,mezic_koopmanism,susuki2011nonlinear,kaiser2017data,surana_observer,peitz2017koopman,mauroy2016global}. This recent work is mainly driven by the ability to approximate the spectrum (i.e., eigenvalues and eigenfunctions) of the Koopman operator from time-series data \cite{rowley2009spectral,DMD_schmitt, EDMD_williams, Umesh_NSDMD}. The data-driven approach for computing the spectrum of the Koopman operator is attractive as it opens up the possibility of employing operator theoretic methods for data-driven control. In fact, research works in \cite{kaiser2017data,peitz2017koopman,MPC1,MPC2} are proposing to develop Koopman operator-based data-driven methods for the design of optimal control and model predictive control for nonlinear and partial differential equations as well. However, the eigenfunctions of the Koopman operator provide only a bilinear representation for the nonlinear control system, and the control of a bilinear system remains a challenging problem. This is in contrast to the linear programming-based framework provided for the optimal control of nonlinear systems using transfer P-F operator in \cite{raghunathan2014optimal,apurba_cdc2018}, where not only linearity but positivity and Markov property of the P-F operator were exploited.

In this paper, we study the more basic problem of designing stabilizing feedback control for a nonlinear system. We use the control Lyapunov function (CLF) approach from nonlinear system theory for the design of stabilizing feedback control \cite{Khalil_book}. While the search for CLFs for a general nonlinear system is a difficult problem, we use a bilinear representation of a nonlinear control system in the Koopman eigenfunction space to search for a CLF for the bilinear system. By restricting the search of CLFs to a class of quadratic Lyapunov functions, we provide a convex programming-based approach for determining the CLF \cite{Boyd_book}. The existence of a CLF provides multiple choices for stabilizing feedback controllers. Simulation results are presented using two different feedback control inputs obtained using the CLF. Another contribution of this paper is in the use of time-series data to obtain a bilinear representation of the nonlinear control system. Hence the approach outlined in this paper can be viewed as a data-driven approach for the design of stabilizing feedback control. 


\section{Preliminaries}\label{section_prelim}

In this section, we present some preliminaries on the Koopman operator and control Lyapunov function-based approach on the design of stabilizing feedback controllers for nonlinear systems. 
\subsection{Koopman Operator}
Consider a continuous-time dynamical system of the form 
\begin{eqnarray}
\dot x=f(x)\label{system}
\end{eqnarray}
where $x\in X\subset \mathbb{R}^n$ and the vector field $f$ is assumed to be continuously differentiable. Let $\phi_t(x)$ be the solution of the system (\ref{system}) starting from initial condition $x$ and at time $t$. Let $\cal O$ be the space of all observables $\varphi: X\to \mathbb{C}$. 
\begin{definition}[Koopman operator] The Koopman semigroup of operators $U_t :{\cal O}\to {\cal O}$ associated with system (\ref{system}) is defined by
\end{definition}
\begin{eqnarray}
[U_t \varphi](x)=\varphi(\phi_t(x)).
\end{eqnarray}
It is easy to observe that the Koopman operator is linear on the space of observables although the underlying dynamical system is nonlinear.  In particular, we have
\[[U_t (\alpha \varphi_1 +\varphi_2)](x)=\alpha [U_t  \varphi_1](x)+[U_t \varphi_2](x).\]

Under the assumption that the function $\varphi$ is continuously differentiable, the semigroup $[U_t \varphi](x)=\rho(x,t)$ can be obtained as the solution of the following partial differential equation
\[\frac{\partial \rho}{\partial t}=f\cdot \nabla \rho=: A_K \rho\]
with initial condition $\rho(x,0)=\varphi(x)$. From the semigroup theory it is known \cite{Lasota} that the operator $A_K$ is the infinitesimal generator for the Koopman operator, i.e., 
\[A_K \rho=\lim_{t\to 0}\frac{U_t \rho-\rho}{t}.\]

The linear nature of Koopman operator allows us to define the eigenfunctions and eigenvalues of this operator as follows.
\begin{definition}[Koopman eigenfunctions] The eigenfunction of Koopman operator is a function $\psi_\lambda\in {\cal O}$ that satisfies
\begin{eqnarray}
[U_t \psi_\lambda ](x)=e^{\lambda t} \psi_\lambda(x)
\end{eqnarray}
for some $\lambda\in \mathbb{C}$. The $\lambda$ is the associated eigenvalue of the Koopman eigenfunction and is assumed to belong to the point spectrum. 
\end{definition}

The spectrum of the Koopman operator is far more complex than simple point spectrum and could include continuous spectrum \cite{Meic_model_reduction}. The eigenfunctions can also be expressed in terms of the infinitesimal generator of the Koopman operator $A_K$ as follows,
$A_K \psi_\lambda =\lambda \psi_\lambda$. The eigenfunctions of Koopman operator corresponding to the point spectrum are smooth functions and can be used as coordinates for linear representation of nonlinear systems.  

\subsection{Feedback Stabilization and Control Lyapunov Functions}
For the simplicity of presentation, we will consider only the case of single input in this paper. All the results carry over to the multi-input case in a straightforward manner. Consider a single input control affine system of the form
\begin{align}
\label{non_lin_sys}
\dot x = f(x)+g(x)u, 
\end{align}
where $x(t) \in \mathbb{R}^n$ denotes the state of the system, $u(t) \in \mathbb{R}$ denotes the single input of the system, and $f, g: \mathbb{R}^n \rightarrow \mathbb{R}^n$ are assumed to be continuously differentiable mappings. We assume that $f(0) = 0$ and the origin is an unstable equilibrium point of the uncontrolled system $\dot x=f(x)$.

The \textit{state feedback stabilization} problem associated with system (\ref{non_lin_sys}) seeks a possible feedback control law
\begin{align*}
u = k(x)
\end{align*}
with $k: \mathbb{R}^n \rightarrow \mathbb{R}$ such that $x = 0$ is asymptotically stable within some domain $\mathcal{D} \subset \mathbb{R}^n$ for the closed-loop system
\begin{align}
\dot x = f(x)+g(x)k(x). \label{closed_loop}
\end{align}

One of the possible approaches for the design of stabilizing feedback controllers for the nonlinear system (\ref{non_lin_sys}) is via control Lyapunov functions that are defined as follows.

\begin{definition}
Let $\mathcal{D} \subset \mathbb{R}^n$ be a neighborhood that contains the equilibrium $x = 0$. A \textit{control Lyapunov function} (CLF) is a continuously differentiable positive definite function $V: \mathcal{D} \rightarrow \mathbb{R}_+$ such that for all $x \in \mathcal{D} \setminus \{0\}$ we have
\end{definition}
\begin{align*}
& \infm_u \ \Big[ V_xf(x) + V_xg(x) u \Big] \\
:= \ & \infm_u \ \left[ \frac{\partial V}{\partial x} \cdot f(x) + \frac{\partial V}{\partial x} \cdot g(x)u \right] < 0.
\end{align*}

It has been shown in \cite{artstein1983stabilization, sontag1989universal} that the existence of a CLF for system (\ref{non_lin_sys}) is equivalent to the existence of a stabilizing control law $u = k(x)$ which is almost smooth everywhere except possibly at the origin $x = 0$.

\begin{thm} [see \cite{astolfi2015feedback}, Theorem 2] \label{clf_thm}
There exists an almost smooth feedback $u = k(x)$, i.e., $k$ is continuously differentiable for all $x \in \mathbb{R}^n \setminus \{0\}$ and continuous at $x = 0$, which globally asymptotically stabilizes the equilibrium $x = 0$ for system (\ref{non_lin_sys}) if and only if there exists a radially unbounded CLF $V(x)$ such that
\begin{enumerate}
\item For all $x \neq 0$, $V_xg(x) = 0$ implies $V_xf(x) < 0$;
\item For each $\varepsilon > 0$, there is a $\delta > 0$ such that $\|x\| < \delta$ implies the existence of a $|u| < \varepsilon$ satisfying $V_xf(x) + V_xg(x) u < 0$.
\end{enumerate}
\end{thm}

In the theorem above, condition 2) is known as the small control property, and it is necessary to guarantee continuity of the feedback at $x \neq 0$. If both conditions 1) and 2) hold, an almost smooth feedback can be given by the so-called Sontag's formula
\begin{align}
\label{sontag} k(x) := \begin{cases} -\frac{V_xf + \sqrt{(V_xf)^2 + (V_xg)^4}}{V_xg} & \text{if } V_xg(x) \neq 0 \\ 0 & \text{otherwise.} \end{cases}
\end{align}

Besides Sontag's formula, we also have several other possible choices to design a stabilizing feedback control law based on the CLF given in Theorem \ref{clf_thm}. For instance, if we are not constrained to any specifications on the continuity or amplitude of the feedback, we may simply choose
\begin{align}\label{control2}
k(x) &:= -K \sgn\big[ V_xg(x) \big] \\
\label{control1}
k(x) &:= -K V_xg(x)
\end{align}
with some constant gain $K > 0$. Then, differentiating the CLF with respect to time along trajectories of the closed-loop (\ref{closed_loop}) yields
\begin{align*}
\dot V & = V_xf(x) - K \big| V_xg(x) \big| \\
\dot V & = V_xf(x) - K [V_xg(x)]^2.
\end{align*}
Hence, by the stabilizability property of condition 1), there must exist some $K$ large enough such that $\dot V < 0$ for all $x \neq 0$, because whenever $V_xf(x) \geq 0$ we have $V_xg(x) \neq 0$.

On the other hand, the CLFs also enjoy some optimality property using the principle of inverse optimal control. In particular, consider the following optimal control problem
\begin{align}
\minimize_u \quad & \int_0^\infty (q(x) + u^\top u)dt \label{cost} \\
\subjt \quad & \dot x = f(x)+g(x)u \nonumber
\end{align}
for some continuous, positive semidefinite function $q: \mathbb{R}^n \rightarrow \mathbb{R}$. Then the modified Sontag's formula
\begin{align}
\label{mod_sontag} k(x) := \begin{cases} -\frac{V_xf + \sqrt{(V_xf)^2 + q(V_xg)^2}}{V_xg} & \text{if } V_xg(x) \neq 0 \\ 0 & \text{otherwise} \end{cases}
\end{align}
builds a strong connection with the optimal control. In particular, if the CLF has level curves that agree in shape with those of the value function associated with cost (\ref{cost}), then the modified Sontag's formula (\ref{mod_sontag}) will reduce to the optimal controller \cite{freeman1996control, primbs1999nonlinear}.

\section{Nonlinear Stabilization and Koopman Operator}

The control Lyapunov function provides a powerful tool for the design of a stabilizing feedback controller which also enjoys some optimality property using the principle of inverse optimality. However, one of the main challenges is providing a systematic procedure to find CLFs. For a general nonlinear system finding a CLF remains a challenging problem. We propose to exploit the linear nature of Koopman operator towards providing a systematic procedure for computing CLFs for nonlinear systems. The main idea is to first transform a nonlinear control system into a bilinear control system using eigenfunctions of the Koopman operator as coordinates. For the bilinear control system, one can search for a CLF in the class of quadratic Lyapunov functions. The search for quadratic CLFs for bilinear systems can be formulated as an optimization problem.

For the simplicity of presentation we will consider in this paper the case of single input control system (\ref{non_lin_sys}). We transform the system Eq. (\ref{non_lin_sys}) in bilinear form using eigenfunctions of the Koopman operator as coordinates. Towards this goal we let 
\[{\bf \Psi}(x):=[\psi_1(x),\ldots, \psi_N(x)]^\top\]
be the Koopman eigenfunctions with eigenvalues $\lambda_i\in \mathbb{C}$, for $i=1,\ldots, N$, and hence $\psi_i$'s are in general complex-valued functions. Utilizing the technique from \cite{surana2016linear}, we can transform these complex eigenfunctions to real as follows. Define
\[\hat{\bf {\Psi}}(x):=[\hat \psi_1(x),\ldots, \hat \psi_N(x)]^\top\] where $\hat \psi_i:=\psi_i$ if $\psi_i$ is a real-valued eigenfunction and $\hat \psi_i:=2 {\rm Re}(\psi)$, $\hat \psi_{i+1}:=-2{\rm Im}(\psi_i)$, if $i$ and $i+1$ are complex conjugate eigenfunction pairs. Consider now the transformation as $\hat {\bf {\Psi}}: \mathbb{R}^n\to \mathbb{R}^N$ as 
\[z=\hat{\bf {\Psi}}(x).\]
Then in this new coordinates system Eq. (\ref{non_lin_sys}) takes the following form
\begin{eqnarray}
\dot z=A z+\frac{\partial \hat {\bf {\Psi}}}{\partial x} g(x) u.\label{bilinear1}
\end{eqnarray}

We now make the following assumption.
\begin{assumption}\label{assume} We assume that $\frac{\partial \hat {\bf {\Psi}}}{\partial x} g$ lies in the span of $\hat {\bf {\Psi}}$, i.e., there exists a constant matrix $B\in \mathbb{R}^{N\times N}$ such that
\end{assumption}
\vspace{4pt}
\[\frac{\partial \hat {\bf {\Psi}}}{\partial x} g = B\hat {\bf {\Psi}}.\]

\begin{remark}
Generally speaking, whether Assumption \ref{assume} holds or not depends on how functions $\hat \psi_1, \cdots, \hat \psi_N$ and $g$ look like. From our simulation results in Section V, where for all the control system examples we set $g$ as a constant vector and use a series of monomials to express $\hat \psi_1, \cdots, \hat \psi_N$, we observe that Assumption \ref{assume} is always true.
\end{remark}

\begin{remark}
In case that Assumption \ref{assume} fails, numerically we may estimate the $B$ matrix via a least squares optimization that is formulated based on the time-series data of $\hat \psi_1, \cdots, \hat \psi_N$ and $g$. Under these circumstances, we can write $\frac{\partial \hat {\bf {\Psi}}}{\partial x} g = B\hat{\bf{\Psi}} + \Delta$ for some $\Delta \in \mathbb{R}^{N}$ as the estimation error. If we are able to further characterize the quantity of $\|\Delta\|$, then the control of dynamics (\ref{bilinear1}) can be reformulated as a robust control problem. We will leave the study of such a robust control problem for future research.
\end{remark}

Using Assumption \ref{assume}, system Eq. (\ref{bilinear1}) can be written as the following bilinear control system
\begin{eqnarray}\label{z_bilin_sys}
\dot z=Az+u Bz.
\end{eqnarray}

Now that we have transformed the original nonlinear system (\ref{non_lin_sys}) into a bilinear form via the Koopman eigenfunctions as coordinates, the complexity of designing a stabilizing feedback can be significantly reduced.

In particular, due to the bilinear structure of the system (\ref{z_bilin_sys}), one can search for a CLF from a class of quadratic positive definite functions with the form $V(z) = z^TPz$. In the sequel, if there exists a quadratic CLF for the bilinear system (\ref{z_bilin_sys}), then we will say that system (\ref{z_bilin_sys}) is \textit{quadratic stabilizable}.

\begin{thm}
\label{bilin_stb}
System (\ref{z_bilin_sys}) is quadratic stabilizable if and only if there exists an $N \times N$ symmetric positive definite $P$ such that for all non-zero $z \in \mathbb{R}^N$ with $z^\top(PA+A^\top P)z \geq 0$, we have $z^\top(PB+B^\top P)z \neq 0$.
\end{thm}

\begin{proof}
Sufficiency $(\Leftarrow)$: Suppose there is a symmetric, positive definite $P$ that satisfies the condition of Theorem \ref{bilin_stb}. We can use it to construct $V(z) = z^\top Pz$ as our Lyapunov candidate function, and the derivative of $V$ with respect to time along trajectories of (\ref{z_bilin_sys}) is given by
\begin{align*}
\dot V & = z^\top P\dot z + \dot z^\top Pz \nonumber \\
& = z^\top(PA+A^\top P)z + uz^\top(PB+B^\top P)z.
\end{align*}
Since for all $z \neq 0$ we have $z^\top(PB+B^\top P)z \neq 0$ when $z^\top(PA+A^\top P)z \geq 0$, we can always find a control input $u(z)$ such that
\begin{align*}
\dot V < 0, \quad \forall z \in \mathbb{R}^N \setminus \{0\}.
\end{align*}
Therefore, $V(z)$ is indeed a CLF for system (\ref{z_bilin_sys}).

Necessity ($\Rightarrow$): We will prove this by contradiction. Suppose that system (\ref{z_bilin_sys}) has a CLF in the form of $V(z) = z^\top Pz$, where $P$ does not satisfy the condition of Theorem \ref{bilin_stb}. That is, there exists some $\bar z \neq 0$ such that ${\bar z}^\top(PA+A^\top P)\bar z \geq 0$ but ${\bar z}^\top(PB+B^\top P)\bar z = 0$. In this case, we have
\begin{align*}
\dot V(\bar z) = {\bar z}^\top(PA+A^\top P)\bar z \geq 0
\end{align*}
for any input $u$, which contradicts the definition of a CLF. This completes the proof.
\end{proof}

Below we formulate a convex optimization problem that attempts to obtain a positive definite $P$ satisfying the condition of Theorem \ref{bilin_stb}, which is
\begin{eqnarray}
\label{opt}
\minimize_{t > 0, \ P = P^T} & \quad t - \gamma{\rm Trace}(P B) \nonumber \\
\subjt & \quad tI - (PA+A^\top P) \succeq 0 \nonumber \\
& \quad c^\text{max} I \succeq P \succeq c^\text{min} I
\end{eqnarray}
where $c^\text{max} > c^\text{min} > 0$, respectively, are two given positive scalars forming bounds for the largest and the least eigenvalues of $P$. The variable $t$ here represents an epigraph form for the largest eigenvalue of $PA+A^\top P$.

Optimization (\ref{opt}) has combined two objectives. On the one hand, we minimize the largest eigenvalue of $PA+A^\top P$. On the other hand, we try to maximize the least singular value of $PB+B^\top P$ the same time. Noticing that it may be difficult to maximize the least singular value of $PB+B^\top P$ directly, we maximize the trace of $PB$ instead and employ a parameter $\gamma > 0$ to balance these two objectives.

\begin{remark}
When an optimal $P^\star$ is solved from (\ref{opt}), we still need to check whether it satisfies the condition of Theorem \ref{bilin_stb} or not. So if one $P^\star$ fails the condition check, then we may tune the parameter $\gamma$ and solve the above optimization again until we obtain a correct $P^\star$. Nevertheless, we observe from simulations (see the multiple examples in our simulation section) that when we choose a $\gamma = 2$, optimization (\ref{opt}) will always yield an optimal $P^\star$ that satisfies the condition of Theorem \ref{bilin_stb}. 
\end{remark}

\begin{remark}
We also need to point out that, compared to searching for a nonlinear CLF for the original nonlinear system (\ref{non_lin_sys}), the procedure for seeking a quadratic CLF for bilinear system (\ref{z_bilin_sys}) becomes quite easier and more systematic. Furthermore, a quadratic CLF for the bilinear system is, in fact, non-quadratic (i.e., contains higher order nonlinear terms) for system (\ref{non_lin_sys}).
\end{remark}

Once a quadratic control Lyapunov function $V(z) = z^\top Pz$ is found for bilinear system (\ref{z_bilin_sys}), we have several choices for designing a stabilizing feedback control law. For instance, applying the control law (\ref{control2}) or (\ref{control1}) we can construct
\begin{align*}
k(z) &= -\beta \sgn\big[ z^\top(PB+B^\top P)z \big] \\
k(z) &= -\beta z^\top(PB+B^\top P)z.
\end{align*}
Moreover, given a positive semidefinite cost $q(z) \geq 0$, we may also apply the inverse optimality property to design an optimal control via Sontag's formula (\ref{mod_sontag}).

\section{Data-driven approximation of Koopman eigenfunctions}
In this paper, we will use Extended Dynamic Mode Decomposition (EDMD) algorithm for the approximation of Koopman eigenfunctions \cite{EDMD_williams}. Given the continuous time system, $\dot x=f(x)$, one can generate the time-series data from the simulation or the experiment as follows
\begin{eqnarray}
\overline X = [x_1,x_2,\ldots,x_M],&\overline Y = [y_1,y_2,\ldots,y_M] \label{data}
\end{eqnarray}
where $x_i\in X$ and $y_i=T(x_i) = f(x_i)\Delta t+x_i\in X$. Now let $\mathcal{H}=
\{h_1,h_2,\ldots,h_N\}$ be the set of dictionary functions or observables. The dictionary functions are assumed to belong to $h_i\in L_2(X,{\cal B},\mu)={\cal G}$, where $\mu$ is some positive measure, not necessarily the invariant measure of $T$. Let ${\cal G}_{\cal H}$ denote the span of ${\cal H}$ such that ${\cal G}_{\cal H}\subset {\cal G}$. The choice of dictionary functions is very crucial and it should be rich enough to approximate the leading eigenfunctions of the Koopman operator. Define vector-valued function $\mathbf{H}:X\to \mathbb{C}^{N}$
\begin{equation}
\mathbf{H}(\boldsymbol{x}):=\begin{bmatrix}h_1(\boldsymbol{x}) & h_2(\boldsymbol{x}) & \cdots & h_N(\boldsymbol{x})\end{bmatrix}^\top.
\end{equation}
In this application, $\mathbf{H}$ is the mapping from state space to function space. Any two functions $\phi$ and $\hat{\phi}\in \mathcal{G}_{\cal H}$ can be written as
\begin{eqnarray}
\phi = \sum_{k=1}^N a_kh_k=\boldsymbol{H^\top a},\quad \hat{\phi} = \sum_{k=1}^N \hat{a}_kh_k=\boldsymbol{H^\top \hat{a}}
\end{eqnarray}
for some coefficients $\boldsymbol{a}$ and $\boldsymbol{\hat{a}}\in \mathbb{C}^N$. Let \[ \hat{\phi}(\boldsymbol{x})=[U_{\Delta t}\phi](\boldsymbol{x})+r\]
where $r\in\mathcal{G}$ is a residual function that appears because $\mathcal{G}_{\cal H}$ is not necessarily invariant to the action of the Koopman operator. To find the optimal mapping which can minimize this residual, let $\bf K$ be the finite dimensional approximation of the Koopman operator $U_{\Delta t}$. Then the  matrix $\bf K$ is obtained as a solution of least square problem as follows 
\begin{equation}\label{edmd_op}
\minimize_{\bf K} \quad \|{\bf G}{\bf K}-{\bf A}\|_F
\end{equation}
where
\begin{eqnarray}\label{edmd1}
{\bf G}=\frac{1}{M}\sum_{m=1}^M \boldsymbol{H}({x}_m)^\top \boldsymbol{H}({x}_m)\nonumber\\
{\bf A}=\frac{1}{M}\sum_{m=1}^M \boldsymbol{H}({x}_m)^\top \boldsymbol{H}({y}_m)
\end{eqnarray}
with ${\bf K},{\bf G},{\bf A}\in\mathbb{C}^{N\times N}$. The optimization problem (\ref{edmd_op}) can be solved explicitly with a solution in the following form
\begin{eqnarray}
{\bf K}_{EDMD}={\bf G}^\dagger {\bf A}\label{EDMD_formula}
\end{eqnarray}
where ${\bf G}^{\dagger}$ denotes the psedoinverse of matrix $\bf G$.
Under the assumption that the leading Koopman eigenfunctions are contained within $\mathcal{G}_{\mathcal{H}}$, the eigenvalues of $\bf K$ are approximations of the Koopman eigenvalues. The right eigenvectors of $\bf K$ can be used then to generate the approximation of Koopman eigenfunctions. In particular, the approximation of Koopman eigenfunction is given by
\begin{equation}\label{EDMD_eigfunc_formula}
\psi_j=\boldsymbol{H} ^\top v_j, \quad j = 1,\ldots,N
\end{equation}
where $v_j$ is the $j$-th right eigenvector of $\bf K$, and $\psi_j$ is the approximation of the eigenfunction of Koopman operator corresponding to the $j$-th eigenvalue. 

For the simulation examples in this paper, we choose the monomials of most degree $D$ as the Koopman dictionary functions. In bilinear system (\ref{z_bilin_sys}), $A\in\mathbb{R}^{N\times N}$ can be written as a block diagonal matrix of Koopman eigenvalues $\lambda_1,\lambda_2,\ldots,\lambda_N$ such that $A_{(i,i)} =\lambda_i$ if $\psi_i$ is real, and 
\begin{align*}
\begin{bmatrix}A_{(i,i)}&A_{(i,i+1)}\\ A_{(i+1,i)}&A_{(i+1,i+1)}\end{bmatrix} =\lvert\lambda_i\rvert\begin{bmatrix}
\cos(\angle\lambda_i)&\sin(\angle\lambda_i)\\-\sin(\angle\lambda_i)&\cos(\angle\lambda_i)
\end{bmatrix}
\end{align*}
if $\psi_i$ and $\psi_{i+1}$ are complex conjugate pairs. Since the real-valued Koopman eigenfunctions satisfy $\boldsymbol{\hat\Psi}(\boldsymbol{x})=V^\top \boldsymbol{H}(\boldsymbol{x})$, we can approximate $B\in\mathbb{R}^{N\times N}$ as
\begin{eqnarray*}
\frac{\partial\boldsymbol{\hat\Psi}}{\partial x}g(x)&=&V^\top\frac{\partial\boldsymbol{H}}{\partial x}g(x)\\
&=&B\boldsymbol{\hat \Psi}(\boldsymbol{x})=BV^\top \boldsymbol{H}(\boldsymbol{x})=\tilde{B}\boldsymbol{H}(\boldsymbol{x}).
\end{eqnarray*}
Since $\frac{\partial{\boldsymbol H}}{\partial x}$ lies in the span of monomial basis functions $\boldsymbol{H}(\boldsymbol{x})$, and  the eigenvector matrix $V$ is invertible, the coefficient matrix $\tilde{B}$ can be found exactly, then $B=\tilde{B}(V^\top)^{-1}$.
 Once $A$ and $B$ matrices are obtained, the controller design problem can be resolved by solving the optimal $P$ satisfying optimization (\ref{opt}).

\section{Simulation results}
In this section, we will present the simulation results for stabilization and optimal control with a variety of unstable continuous-time dynamical systems.




\underline{\it Simple pendulum oscillator}

Consider a controlled 2D pendulum oscillator system given as follows
\begin{eqnarray}\label{linear_sys}
\dot{x}_1 &=& x_2\\\nonumber
\dot{x}_2 &=& 0.01x_2 - sin(x_1)+u
\end{eqnarray}
where $x=[\theta\;\dot\theta]\in\mathbb{R}^2$ and $u\in\mathbb{R}$ is the single input. The nonlinear system without control has a unique unstable equilibrium point at the origin. In this example, we will use the control input  defined in the formula (\ref{control1}) with $K = 10$. The dictionary function $\boldsymbol{H}(\boldsymbol{x})$ is chosen as monomials of most degree $D=5$ (i.e., 21 monomials).

For the comparison with the standard approaches, the classic LQR controller is chosen to stabilize the linearized pendulum system at the origin point. By choosing $Q = \begin{bmatrix} 1&0\\0&0\end{bmatrix}$ and $R = 1$, the optimal $K$ matrix is obtained as, $K = [0.4142\;0.9202]$. 
For the closed-loop simulation, we randomly choose initial points within $[-1,1]\times[-1,1]$ and solve the closed-loop system with \texttt{ode45} solver in \textbf{MATLAB}. Both LQR controller and our controller are applied to the nonlinear pendulum system based on the feedback control $u=-Kx$.

In Fig.~\ref{fig:A-xy}, the closed-loop trajectory with the LQR controller is converging to the origin in 8 seconds, while the closed-loop trajectory with the data-driven designed controller (\ref{control1}) arrives at the origin within 4 seconds. The controller designed by our approach also performs a shorter controlled path to the origin as shown in Fig.~\ref{fig:A-xy}.

\underline{\it Van der Pol oscillator}

The Van der Pol oscillator is described by the following equations
\begin{eqnarray}\label{vanderpol_sys}
\dot{x}_1 &=& x_2\\\nonumber
\dot{x}_2 &=& (1-x_1^2)x_2-x_1+u
\end{eqnarray}
where $x\in\mathbb{R}^2$ and $u\in\mathbb{R}$ is the single input. With $u=0$, the vector field of the Van der Pol oscillator has a limit cycle and an unstable equilibrium point at the origin. In this example, we will apply different types of control $u$ based on the CLF.

For the approximation of Koopman operator and eigenfunctions, we are using the time-series data with $T_{final} = 10$, $\Delta t= 0.0001$ (i.e., $10^5$ time-series data samples). In this example, we are using  the monomials of most degree $D=5$ (i.e., 21 monomials) as the dictionary functions $\boldsymbol H(\boldsymbol x)$.

For the closed-loop simulation, we randomly choose the initial points within $[-3,3]\times[-4,4]$ and solve the closed-loop system with \texttt{ode45} solver in \textbf{MATLAB}. By using the controller $1$ defined in equation (\ref{control1}) with $K=10$, the  closed-loop trajectory is stabilized to the origin within $10$ secs, while the open-loop trajectory from the same initial condition converges to the limit cycle, as shown in the Fig.~\ref{fig:B-xy}.

Without change of initial conditions, the closed-loop simulation applied with controller $2$ defined in the formula (\ref{mod_sontag}) is shown in Fig.~\ref{fig:B-xy2}. The closed-loop trajectory converges to the origin within 1 second. 

\begin{figure}[tbp]
\centering
\includegraphics[width=\linewidth]{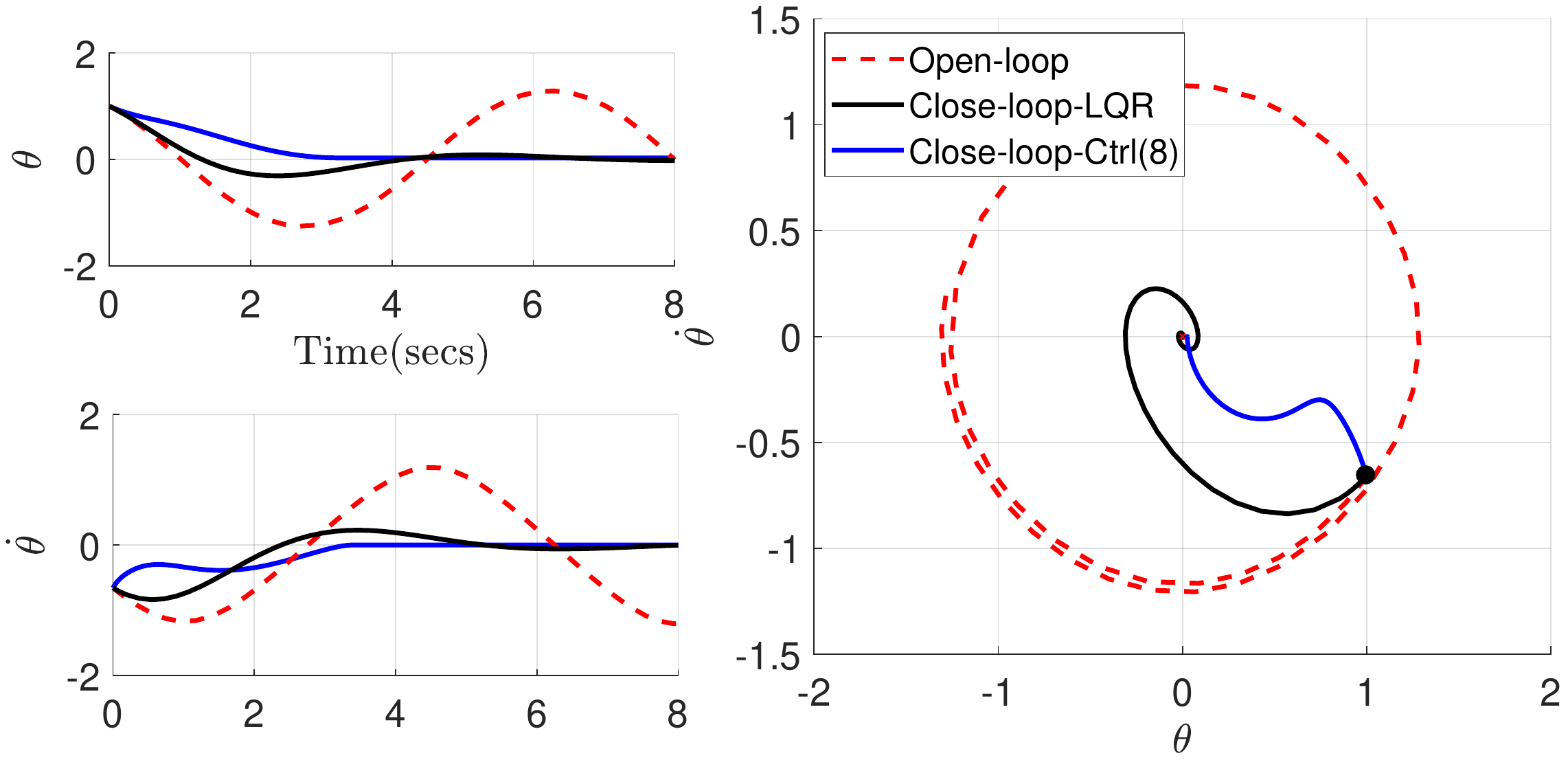}
\caption{Closed-loop and open-loop trajectories for the 2D pendulum system}
 \label{fig:A-xy}
\end{figure}

\begin{figure}[tbp]
\centering
\includegraphics[width=\linewidth]{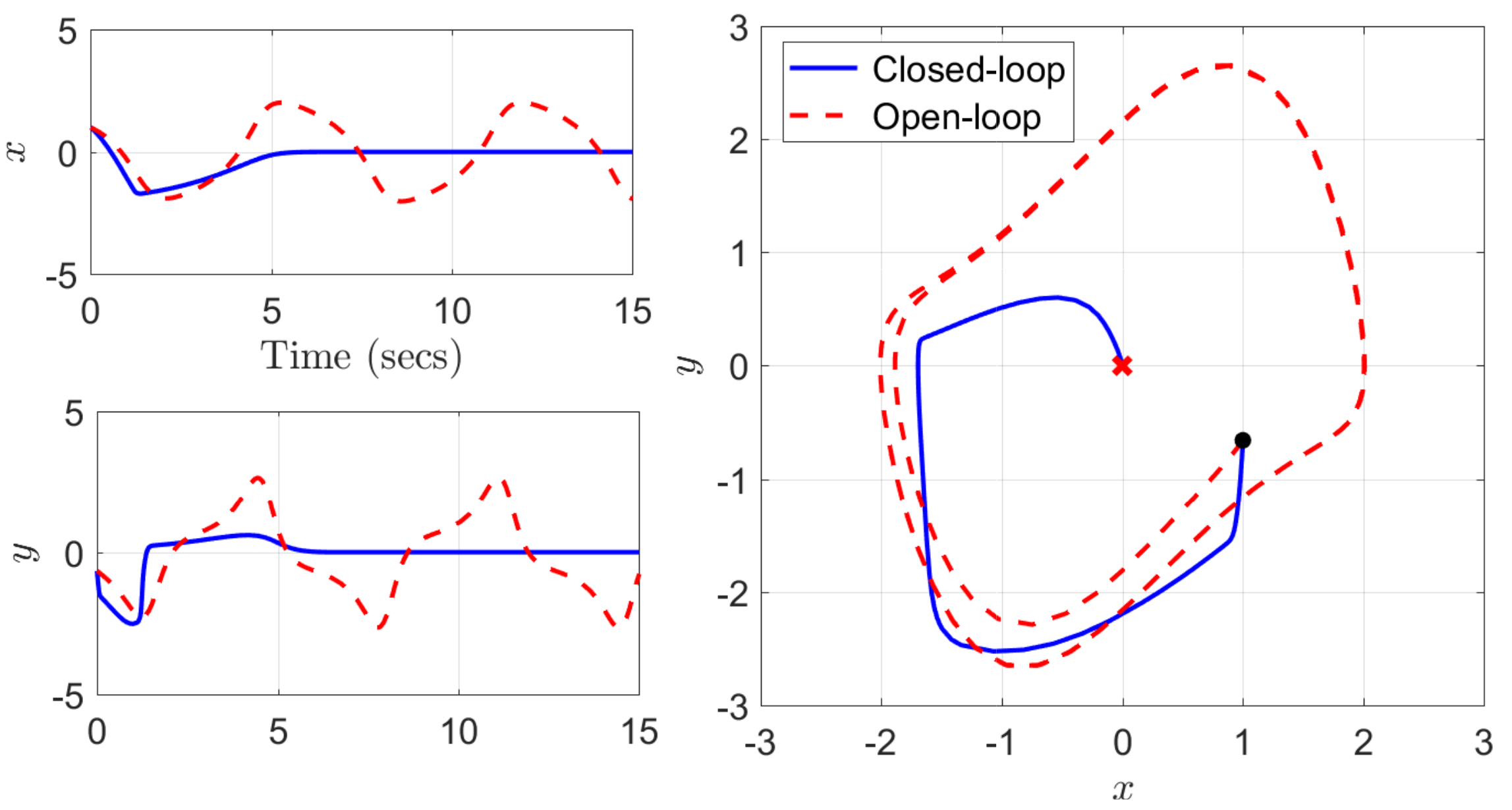}
\caption{Closed-loop and open-loop trajectories for the Van der Pol oscillator with $u$ in (\ref{control1})}
\label{fig:B-xy}
\end{figure}

\begin{figure}[tbp]
\centering
\includegraphics[width=\linewidth]{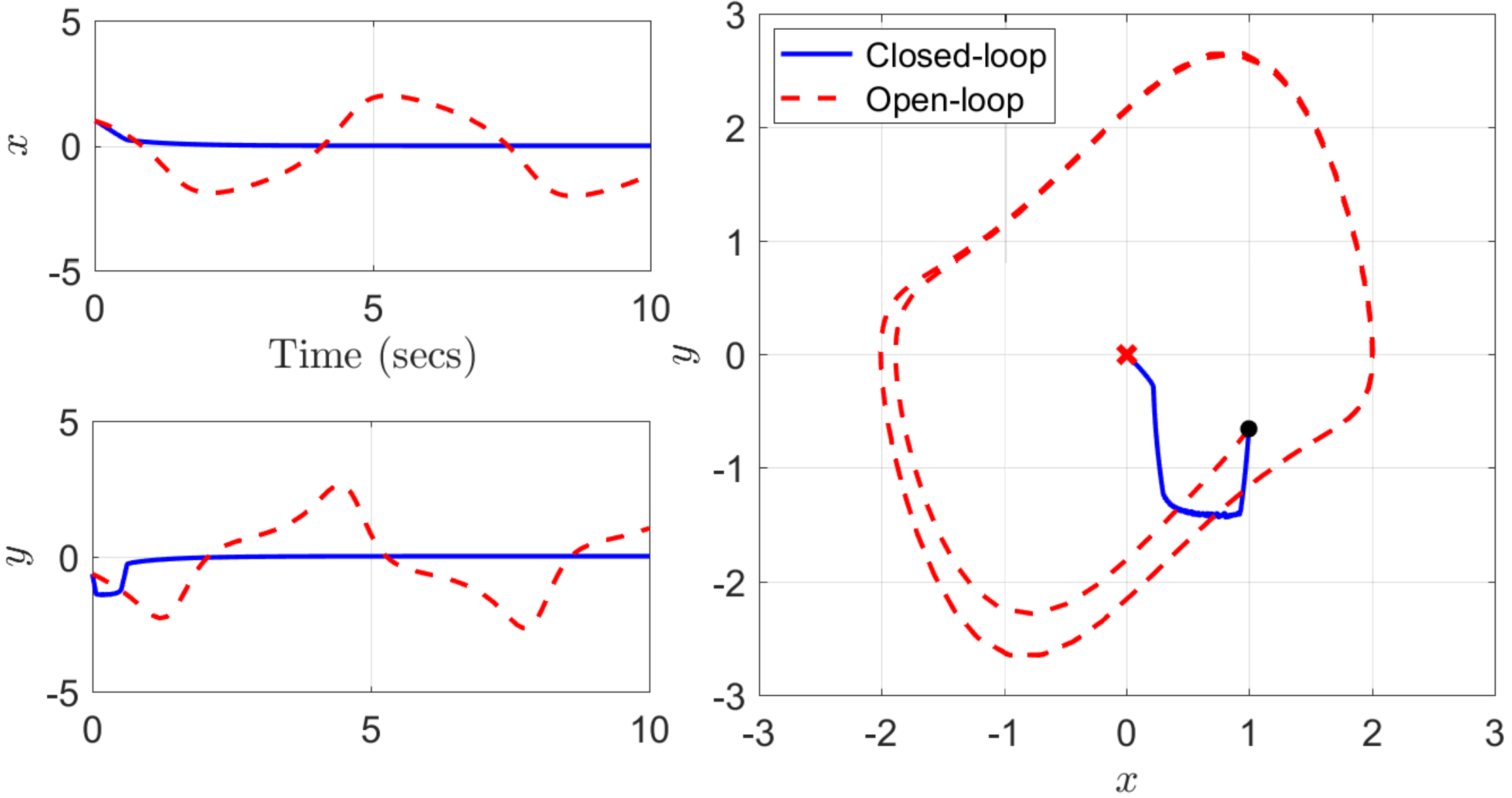}
\caption{Closed-loop and open-loop trajectories for the Van der Pol oscillator with $u$ in (\ref{mod_sontag})}
 \label{fig:B-xy2}
\end{figure}


\begin{figure}[tbp]
\centering 
\includegraphics[width=\linewidth]{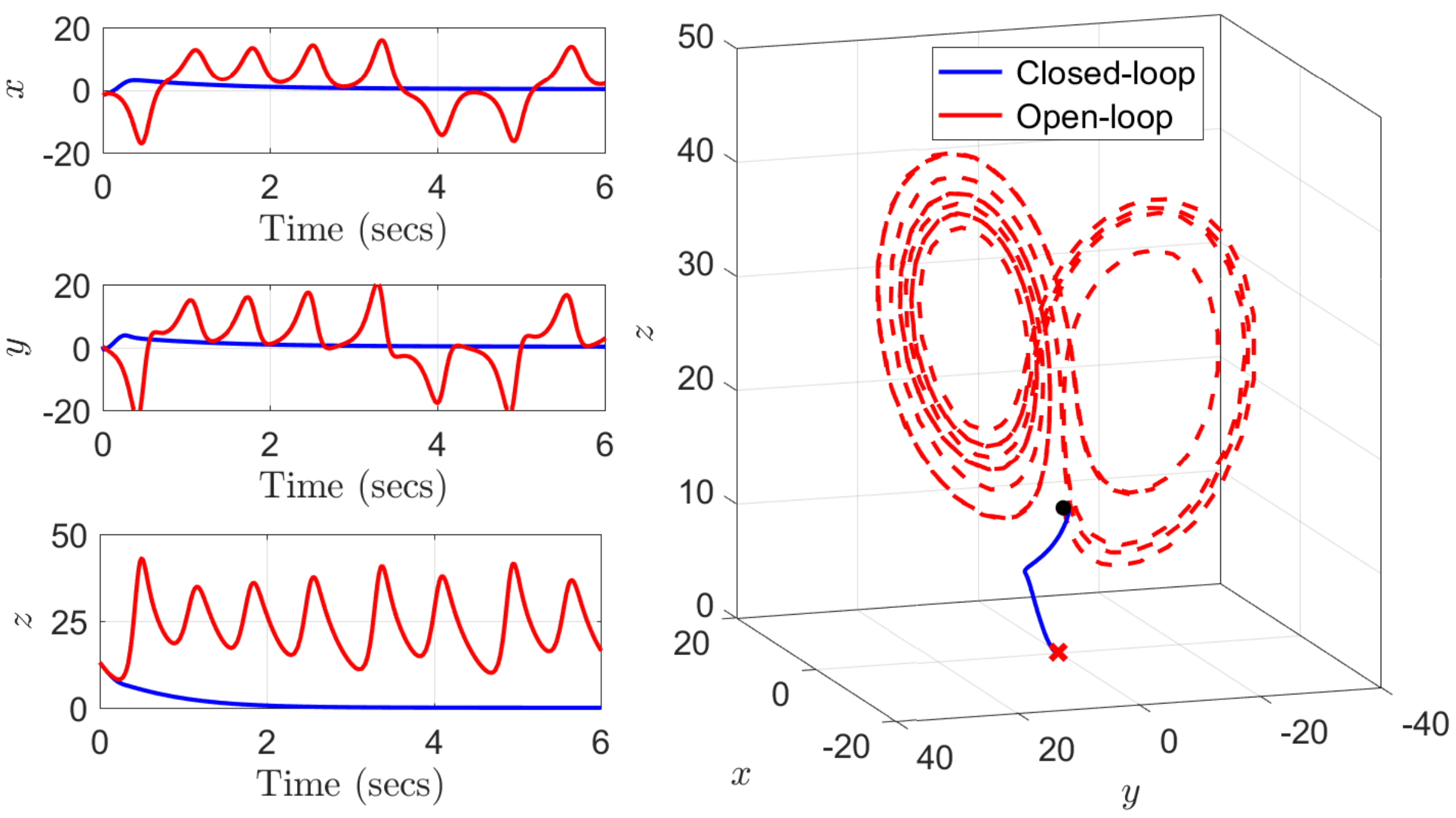}
\caption{Closed-loop and open-loop trajectories for the Lorenz attractor with controller $u$ in (\ref{control1})}
 \label{fig:E-xyz} \vspace{-6pt}
\end{figure}

\underline{\it Lorenz attractor}

The control dynamics for the Lorentz system is given by 
\begin{eqnarray}\label{L}
\dot{x}_1 &=& \sigma(x_2-x_1)\\\nonumber
\dot{x}_2 &=& x_1(\rho-x_3)-x_2+u\\\nonumber
\dot{x}_3 &=& x_1x_2 - \beta x_3
\end{eqnarray}
where $x\in\mathbb{R}^3$ and $u\in\mathbb{R}$ is the single input. When $\rho=28$, $\sigma=10$, $\beta=\frac{8}{3}$, the Lorenz system has chaotic solutions, which means almost all initial points tend to an invariant set except the origin. In this 3D example, the monomial of most degree $5$ is used as the dictionary functions. It is a more challenging problem to stabilize a 3D system to the origin with control $u$ in (\ref{control1}).

For the approximation of Koopman operator and eigenfunctions, the time-series data and time interval are chosen as $T_{final} = 5$ and $\Delta t= 0.001$.
For the closed-loop simulation, we randomly initialize the points within $[-5,5]\times[-5,5]\times[0,20]$ and solve the closed-loop system with \texttt{ode15s} solver in \textbf{MATLAB}. 
In Fig.~\ref{fig:E-xyz}, the open-loop trajectories tend to the invariant set of Lorenz attractor without control. It is shown that all the closed-loop trajectories with the controller are converging to the origin within 2 seconds, which means the controlled system is stabilized perfectly.
\vspace{-0.05in}
\section{Conclusion}
In this paper, we provided a systematic approach for the design of stabilizing feedback control for nonlinear systems. The proposed systematic approach relied on a bilinear representation of the nonlinear control system in the Koopman eigenfunction space. The stabilization problem for the bilinear control system was solved using the control Lyapunov function approach. A convex optimization-based formulation was proposed for the search of quadratic CLFs for the bilinear system. Simulation results were also presented and verify the main findings of the paper. 
\bibliographystyle{ieeetr}
\bibliography{ref,ref1}

\end{document}